\RequirePackage{fix-cm}
\documentclass[smallextended,envcountsame]{svjour3} 
\smartqed  
\usepackage{graphicx, amssymb}
\usepackage{times,a4wide,mathrsfs,tikzsymbols}
\usepackage{amsmath}
\usepackage{amsfonts}

\newcommand{\C}{\mathbb{C}}

\newcommand{\QQ}{\mathbb{Q}}

\newcommand{\PP}{\mathbb{P}}

\newcommand{\MM}{\mathcal M}

\newcommand{\wt}{\widetilde}
\newcommand{\rom}{\romannumeral}

\DeclareMathOperator{\ide}{id}
\DeclareMathOperator{\ima}{Im}

\newtheorem{convention}{Conventions}

\newtheorem{nonumbering}{Theorem}

\newtheorem{nonumberingc}{Corollary}

\newtheorem{nonumberingt}{Acknowledgements}

\begin{document}

\title{Some Calabi--Yau fourfolds verifying Voisin's conjecture}

\author{Robert Laterveer}

\institute{CNRS - IRMA, Universit\'e de Strasbourg \at
              7 rue Ren\'e Descartes \\
              67084 Strasbourg cedex\\
              France\\
              \email{laterv@math.unistra.fr}   }

\date{Received: date / Accepted: date}

\maketitle

\begin{abstract} Motivated by the Bloch--Beilinson conjectures, Voisin has made a conjecture concerning zero--cycles on self--products of Calabi--Yau varieties. This note contains some examples of Calabi--Yau fourfolds verifying Voisin's conjecture.
\end{abstract}

\keywords{Algebraic cycles \and Chow groups \and motives \and Voisin's conjecture}

\subclass{14C15, 14C25, 14C30.}

\section{Introduction}

Let $X$ be a smooth projective variety over $\C$, and let $A^i(X):=CH^i(X)_{\QQ}$ denote the Chow groups of $X$ (i.e. the groups of codimension $i$ algebraic cycles on $X$ with $\QQ$--coefficients, modulo rational equivalence). As is well--known, the world of algebraic cycles is densely inhabited by open problems \cite{B}, \cite{J2}, \cite{Kim}, \cite{Vo}. One of these open problems is the following conjecture formulated by Voisin (which can be seen as a version of Bloch's conjecture for varieties of geometric genus one):

\begin{conjecture}[Voisin \cite{V9}]\label{conjvois} Let $X$ be a smooth projective complex variety of dimension $n$ with $h^{j,0}(X)=0$ for $0<j<n$ and $p_g(X)=1$.  
For any zero--cycles $a,a^\prime\in A^n_{}(X)$ of degree $0$, we have
  \[ a\times a^\prime=(-1)^n a^\prime\times a\ \ \ \hbox{in}\ A^{2n}(X\times X)\ .\]
  (Here $a\times a^\prime$ is short--hand for the cycle class $(p_1)^\ast(a)\cdot(p_2)^\ast(a^\prime)\in A^{2n}(X\times X)$, where $p_1, p_2$ denote projection on the first, resp. second factor.)
   \end{conjecture}
   
 So far, conjecture \ref{conjvois} is wide open for a general $K3$ surface (on the positive side, cf. \cite{V9}, \cite{moi}, \cite{des}, \cite{tod} for some cases where this conjecture is verified).

  The main result of this note gives a series of examples in dimension $4$ verifying Voisin's conjecture:
  
  \begin{nonumbering}[=theorem \ref{main}] Let $S$ be a $K3$ surface obtained as (a desingularization of) a double plane branched along the  union of a smooth quartic and a smooth quadric. Let $\iota\colon S\to S$ be the covering involution. Let $X$ be one of the following:
 
 \noindent
 (\rom1) a smooth Calabi--Yau fourfold obtained as a crepant resolution of the quotient $S^{[2]}/\iota^{[2]}$;

 \noindent
 (\rom2) a smooth Calabi--Yau fourfold obtained as a crepant resolution of the quotient $(S\times S)/(\iota\times\iota)$.
  
  Then any $a,a^\prime\in A^4_{}(X)$ of degree $0$ satisfy the equality
   \[ a\times a^\prime =a^\prime\times a \ \ \ \hbox{in}\ A^8(X\times X)\ .\]
  \end{nonumbering} 
     
  It should be mentioned that theorem \ref{main} is not empty of content: in both cases (\rom1) and (\rom2), crepant resolutions of the quotient do indeed exist, thanks to work of Camere--Garbagnati--Mongardi \cite{CGM} (cf. subsection \ref{squot} below).  
  
  As a consequence of theorem \ref{main}, a certain instance of the generalized Hodge conjecture is verified:
  
  \begin{nonumberingc}[=corollary \ref{ghc}] Let $X$ be a Calabi--Yau fourfold as in theorem \ref{main}. Then the Hodge substructure
    \[ \wedge^2 H^4(X)\ \subset\ H^8(X\times X)\]
    is supported on a divisor.
    \end{nonumberingc}  
     
  The argument proving theorem \ref{main} is very simple, and goes as follows: thanks to the notion of {\em multiplicative Chow--K\"unneth decomposition\/}
  developed by Shen--Vial \cite{SV}, we are reduced to proving Voisin's conjecture holds for $K3$ surfaces $S$ as in theorem \ref{main}. This statement for $S$ was already known \cite{moi}.

  \vskip0.6cm

\begin{convention} In this note, the word {\sl variety\/} will refer to a reduced irreducible scheme of finite type over $\C$. 

All Chow groups will be with rational coefficients: For any variety $X$, we will denote by $A_j(X)$ the Chow group of $j$--dimensional cycles on $X$ with $\QQ$--coefficients.
For $X$ smooth of dimension $n$ the notations $A_j(X)$ and $A^{n-j}(X)$ will be used interchangeably. 

The notations 
$A^j_{hom}(X)$ and $A^j_{AJ}(X)$ will be used to indicate the subgroups of 
homologically, resp. Abel--Jacobi trivial cycles.
For a morphism $f\colon X\to Y$, we will write $\Gamma_f\in A_\ast(X\times Y)$ for the graph of $f$.
The contravariant category of Chow motives (i.e., pure motives with respect to rational equivalence as in \cite{Sch}, \cite{MNP}) will be denoted $\MM_{\rm rat}$. 

We will write $H^j(X)$ to indicate singular cohomology $H^j(X,\QQ)$.
\end{convention}

 \section{Preparatory material}

 \subsection{Quotient varieties}
 
 \begin{definition} A {\em projective quotient variety\/} is a variety $Y=X/G$,
 where $X$ is a smooth projective variety and $G$ is a finite group of automorphisms of $X$.
  \end{definition}
  
 \begin{proposition}[Fulton \cite{F}]\label{opquot} Let $Y$ be a projective quotient variety of dimension $n$. Let $A^\ast(Y)$ denote the operational Chow cohomology ring. The natural map
   \[ A^i(Y)\ \to\ A_{n-i}(Y) \]
   is an isomorphism for all $i$.
   \end{proposition}
   
   \begin{proof} This is \cite[Example 17.4.10]{F}.
      \end{proof}

\begin{remark} It follows from proposition \ref{opquot} that the formalism of correspondences goes through unchanged for projective quotient varieties (this is also noted in \cite[Example 16.1.13]{F}). 
  \end{remark}

 \subsection{MCK decomposition}
 
  \begin{definition}[Murre \cite{Mur}]\label{ck} Let $X$ be a smooth projective variety of dimension $n$. We say that $X$ has a {\em CK decomposition\/} if there exists a decomposition of the diagonal
   \[ \Delta_X= \pi_0+ \pi_1+\cdots +\pi_{2n}\ \ \ \hbox{in}\ A^n(X\times X)\ ,\]
  such that the $\pi_i$ are mutually orthogonal idempotents and $(\pi_i)_\ast H^\ast(X)= H^i(X)$.
  
  (NB: ``CK decomposition'' is short--hand for ``Chow--K\"unneth decomposition''.)
\end{definition}

\begin{remark} The existence of a CK decomposition for any smooth projective variety is part of Murre's conjectures \cite{Mur}, \cite{J2}. 
\end{remark}

\begin{definition}[Shen--Vial \cite{SV}]\label{mck} Let $X$ be a smooth projective variety of dimension $n$. Let $\Delta_X^{sm}\in A^{2n}(X\times X\times X)$ be the class of the small diagonal
  \[ \Delta_X^{sm}:=\bigl\{ (x,x,x)\ \vert\ x\in X\bigr\}\ \subset\ X\times X\times X\ .\]
  An MCK decomposition is a CK decomposition $\{\pi_i\}$ of $X$ that is {\em multiplicative\/}, i.e. it satisfies
  \[ \pi_k\circ \Delta_X^{sm}\circ (\pi_i\times \pi_j)=0\ \ \ \hbox{in}\ A^{2n}(X\times X\times X)\ \ \ \hbox{for\ all\ }i+j\not=k\ .\]
  
 (NB: ``MCK decomposition'' is short--hand for ``multiplicative Chow--K\"unneth decomposition''.) 
  \end{definition}
  
  \begin{remark} The small diagonal (seen as a correspondence from $X\times X$ to $X$) induces the {\em multiplication morphism\/}
    \[ \Delta_X^{sm}\colon\ \  h(X)\otimes h(X)\ \to\ h(X)\ \ \ \hbox{in}\ \MM_{\rm rat}\ .\]
 Suppose $X$ has a CK decomposition
  \[ h(X)=\bigoplus_{i=0}^{2n} h^i(X)\ \ \ \hbox{in}\ \MM_{\rm rat}\ .\]
  By definition, this decomposition is multiplicative if for any $i,j$ the composition
  \[ h^i(X)\otimes h^j(X)\ \to\ h(X)\otimes h(X)\ \xrightarrow{\Delta_X^{sm}}\ h(X)\ \ \ \hbox{in}\ \MM_{\rm rat}\]
  factors through $h^{i+j}(X)$.
  It follows that if $X$ has an MCK decomposition, then setting
    \[ A^i_{(j)}(X):= (\pi^X_{2i-j})_\ast A^i(X) \ ,\]
    one obtains a bigraded ring structure on the Chow ring: that is, the intersection product has the property that 
    \[  \ima \Bigl(A^i_{(j)}(X)\otimes A^{i^\prime}_{(j^\prime)}(X) \xrightarrow{\cdot} A^{i+i^\prime}(X)\Bigr)\ \subset\  A^{i+i^\prime}_{(j+j^\prime)}(X)\ .\]
    It is expected that for any $X$ with an MCK decomposition, one has
    \[ A^i_{(j)}(X)\stackrel{??}{=}0\ \ \ \hbox{for}\ j<0\ ,\ \ \ A^i_{(0)}(X)\cap A^i_{hom}(X)\stackrel{??}{=}0\ ;\]
    this is related to Murre's conjectures B and D \cite{Mur}.

  The property of having an MCK decomposition is severely restrictive, and is closely related to Beauville's ``(weak) splitting property'' \cite{Beau3}. For more ample discussion, and examples of varieties with an MCK decomposition, we refer to \cite[Chapter 8]{SV}, as well as \cite{V6}, \cite{SV2}, \cite{FTV}.
    \end{remark}

In this note, we will rely on the following result:

\begin{theorem}[Shen--Vial \cite{SV}]\label{hilb2} Let $S$ be a $K3$ surface. Then $S$ and $S^2$ and the Hilbert scheme $S^{[2]}$ have an MCK decomposition. Moreover, the bigraded ring structure on the Chow ring $A^\ast_{(\ast)}(S^{[2]})$ coincides with the Fourier decomposition as defined in \cite{SV}.
\end{theorem}

\begin{proof} For $S$, this is \cite[Example 8.17]{SV}. The statement for $S^2$ follows because the property of having an MCK decomposition is stable under products \cite[Theorem 8.6]{SV}. The statement for $S^{[2]}$ is \cite[Theorem 13.4]{SV} (or alternatively \cite{V6}). Finally, the relation with the Fourier decomposition for $A^\ast(S^{[2]})$ is \cite[Theorem 15.8]{SV}.

\end{proof}

 \subsection{Calabi--Yau quotients}\label{squot}
 
 \begin{definition} A {\em Calabi--Yau variety\/} is a smooth projective variety $X$ such that the canonical bundle $K_X$ is numerically trivial, and $h^{j,0}(X)=0$ for all $0<j<n$.
 \end{definition}
 
 \begin{definition}[Beauville \cite{Beau0}, \cite{Beau1}] A {\em hyperk\"ahler variety\/} is a smooth projective simply--connected variety $X$, such that $H^{2,0}(X)$ is generated by a symplectic form.
 \end{definition}

 \begin{theorem}[Camere--Garbagnati--Mongardi \cite{CGM}]\label{cgm} Let $Y$ be a hyperk\"ahler fourfold, and $\iota\colon Y\to Y$
 a non--symplectic involution. Assume that all irreducible components of the fixed locus of $\iota$ have dimension $2$. Then
 there exists a resolution of singularities
   \[ X\ \to\ Y/\iota\]
   such that $X$ is a Calabi--Yau variety.
 \end{theorem}
 
 \begin{proof} This is the dimension $4$ case of \cite[Theorem 3.7]{CGM}.
 \end{proof}
 
 \begin{corollary}[Camere--Garbagnati--Mongardi \cite{CGM}]\label{cgm2} Let $S$ be a $K3$ surface obtained as (a desingularization of) a double plane branched along the  union of a smooth quartic and a smooth quadric. Let $Y$ be the Hilbert scheme $Y:=S^{[2]}$. Let $\iota^{[2]}$ be the natural involution of $Y$ induced by the covering involution $\iota$ of $S$. 
 Then
 there exists a resolution of singularities
   \[ X\ \to\ Y/\iota\]
   such that $X$ is a Calabi--Yau variety.
   \end{corollary}
   
 \begin{proof} The condition in theorem \ref{cgm} on the dimension of the fixed locus is verified by what is detailed in \cite[Section 5.1]{CGM}.
   \end{proof}

  \begin{remark}\label{dom} Let $S$ and $X$ be a $K3$ surface resp. a Calabi--Yau fourfold as in corollary \ref{cgm2}. As explained in \cite[Section 5.2]{CGM}, there exists a resolution
   \[ X_2\ \to\ (S\times S)/(\iota\times \iota)\ ,\]
   such that $X_2$ is also a Calabi--Yau variety, and there is a rational $2:1$ map $X_2\dashrightarrow X$.
  \end{remark}

 \begin{remark}\label{pq} In \cite[Section 5.1 and Appendix 8]{CGM}, the Hodge numbers of Calabi--Yau resolutions $X$ as in corollary \ref{cgm2} are computed. Likewise, the Hodge numbers of the Calabi--Yau resolutions of the quotient $(S\times S)/(\iota\times\iota)$ are computed in \cite[Section 5.2]{CGM}.
 \end{remark}

 \section{Main result}
 
 \begin{theorem}\label{main} Let $S$ be a $K3$ surface obtained as (a desingularization of) a double plane branched along the  union of a smooth quartic and a smooth quadric. Let $\iota\colon S\to S$ be the covering involution. Let $X$ be one of the following:
 
 \noindent
 (\rom1) a smooth Calabi--Yau fourfold obtained as a crepant resolution of the quotient $S^{[2]}/\iota^{[2]}$;

 \noindent
 (\rom2) a smooth Calabi--Yau fourfold obtained as a crepant resolution of the quotient $(S\times S)/(\iota\times\iota)$.
  
  Then any $a,a^\prime\in A^4_{hom}(X)$ satisfy the equality
   \[ a\times a^\prime - a^\prime\times a =0\ \ \ \hbox{in}\ A^8(X\times X)\ .\]
   \end{theorem}

  \begin{proof} It will suffice to treat case (\rom2). Indeed, let $X_1$ be a smooth fourfold as in case (\rom1). As we have seen (remark \ref{dom}), $X_1$ is rationally dominated by a fourfold $X_2$ which is as in case (\rom2). The rational map $\phi\colon X_2\dashrightarrow X_1$ gives rise to a commutative diagram
   \[  \begin{array}[c]{ccc}
      A^4_{hom}(X_2)\otimes A^4_{hom}(X_2) & \xrightarrow {} &   A^8_{}(X_2\times X_2)  \\
     \ \ \ \ \ \  \uparrow {\scriptstyle (\phi^\ast,\phi^\ast)}&&\ \ \ \ \ \  \uparrow {\scriptstyle (\phi\times\phi)^\ast}\\
       A^4_{hom}(X_1)\otimes A^4_{hom}(X_1) & \xrightarrow {} &   A^8_{}(X_1\times X_1)  \\
       \end{array}\]
    (Here the horizontal arrows are defined as $(a,a^\prime)\mapsto a\times a^\prime - a^\prime\times a$.) Since the vertical arrows are injective, we are reduced to proving the statement for $X_2$.
    
 Let us now treat case (\rom2), i.e. let us suppose $X=X_2$ is obtained as a crepant resolution of 
  \[ Z:=  (S\times S)/(\iota\times\iota)\ ,\]
  where $S$ is a double plane as in the theorem.

  \begin{lemma}\label{blowup} The morphism $f\colon X\to Z$ induces isomorphisms
   \[ \begin{split} f^\ast\colon\ \ \ A_0(Z)\ &\xrightarrow{\cong}\ A^4(X)\ ,\\
            (f\times f)^\ast\colon\ \ \ A_0(Z\times Z)\ &\xrightarrow{\cong}\ A^8(X\times X)\ ,\\
            \end{split}   \]
  \end{lemma} 
  
  \begin{proof} As $Z$ is a projective quotient variety, proposition \ref{opquot} applies, i.e. the natural map $A^4(Z)\to A_0(Z)$ (from operational Chow cohomology to the Chow group) is an isomorphism. Let $T\subset Z$ denote the singular locus, and $E:=f^{-1}(T)\subset X$ the exceptional divisor. Thanks to \cite{Kim0}, there exist exact sequences
   \[ 0\to A^i(Z)\to A^i(X)\oplus A^i(T)\to A^i(E) \]
   for all $i$. Taking $i=4$, and noting that $A^4(T)=A^4(E)=0$ for dimension reasons, we obtain an isomorphism
   \[ f^\ast \colon\ \ \ A_0(Z)=A^4(Z)\ \xrightarrow{\cong}\ A^4(X)\ .\]
   The argument for $Z\times Z$ is only notationally different.
  \end{proof}
  
  The resolution morphism $f\colon X\to Z$ induces a commutative diagram    
  \[  \begin{array}[c]{ccc}
      A^4_{hom}(X)\otimes A^4_{hom}(X) & \xrightarrow {} &   A^8_{}(X\times X)  \\
     \ \ \ \ \ \  \uparrow {\scriptstyle (f^\ast,f^\ast)}&&\ \ \ \ \ \  \uparrow {\scriptstyle (f\times f)^\ast}\\
       A^4_{hom}(Z)\otimes A^4_{hom}(Z) & \xrightarrow {} &   A^8_{}(Z\times Z)  \\
       \end{array}\]
       
    Since vertical arrows are isomorphisms (lemma \ref{blowup}), we are reduced to proving the statement for $Z$, i.e. we need to prove that for all $a,a^\prime\in A^4_{hom}(Z)$, one has equality
    \begin{equation}\label{Z}   a\times a^\prime - a^\prime\times a =0\ \ \ \hbox{in}\ A^8(Z\times Z)\ .\end{equation}
    
    Next, we reduce to $S\times S$ using the following lemma:
    
    \begin{lemma}\label{prod} Let $q\colon S\times S\to Z$ denote the quotient morphism. Then
    \[ \ima\Bigl( A^4_{hom}(Z)\xrightarrow{q^\ast} A^4(S\times S)\Bigr) = A^4_{(4)}(S\times S)\ .\]
    \end{lemma}
    
    \begin{proof} First, we observe that
      \[  \ima\Bigl( A^4_{hom}(Z)\xrightarrow{q^\ast} A^4(S\times S)\Bigr) \ \subset\ A^4_{hom}(S\times S)=\bigoplus_{j\in\{2,4\}}   A^4_{(j)}(S\times S)  \ .\]
      Next, we observe that
      \[  \ima\Bigl( A^4_{}(Z)\xrightarrow{q^\ast} A^4(S\times S)\Bigr) = A^4_{}(S\times S)^{(\iota\times\iota)}\ \]
      (where $A^\ast()^{(\iota\times\iota)}$ denotes cycles invariant under ${(\iota\times\iota)}$).
      
     The following claim concludes the proof of lemma \ref{prod}:
     
     \begin{claim}\label{anti} We have
     \[ (\iota\times\iota)^\ast =\begin{cases} \ide\ \colon\ & A^4_{(4)}(S\times S)\ \to A^4(S\times S)\ ,\\
                      -\ide\ \colon\ & A^4_{(2)}(S\times S)\ \to A^4(S\times S)\ .\\
                      \end{cases}\]
                      \end{claim}
                      
         To prove the claim, we note that
         \begin{equation}\label{iota}
          \iota^\ast= \begin{cases} -\ide \colon\ & A^2_{(2)}(S)\ \to\ A^2(S)\ ,\\
                                                      \ide \colon\ & A^2_{(0)}(S)\ \to\ A^2(S)\ .\\
                                                      \end{cases}\end{equation}   
                            Indeed, let $p\colon S\to V:=S/\iota$ denote the quotient morphism. Then we have
               \[ p^\ast p_\ast= \ide + \iota^\ast\colon\ \ \ A^i(S)\ \to\ A^i(S)  \ \ \ \forall i\ .\]
               On the other hand, 
               \[  p^\ast p_\ast=0\colon\ \ \ A^2_{hom}(S)\ \to\ A^2_{hom}(S)\ ,\]
               because $V$ is birational to $\PP^2$ and so $A^2_{hom}(V)=0$. This proves the first line of (\ref{iota}).
               For the second line of (\ref{iota}), we note that $A^2_{(0)}(S)$ is one--dimensional and spanned by the class of the intersection $D\cdot D^\prime$, where $D,D^\prime$ are any non--zero effective divisors \cite{BV}. Letting $B\subset S$ denote 
      the inverse image of the ramification locus, and setting $D=D^\prime=B$ establishes the second line of (\ref{iota}).

      The equalities (\ref{iota}) suffice to prove the claim, since 
      \[ A^4_{(j)}(S\times S)=\bigoplus_{j_1+j_2=j} A^2_{(j_1)}(S)\otimes A^2_{(j_2)}(S) \]
      (because of the bigraded ring structure), and
      \[ A^2(S)=A^2_{(0)}(S)\oplus A^2_{(2)}(S)\]
      (because $\pi_1^S=0$). This proves the claim, and hence lemma \ref{prod}.
       \end{proof}  
       
     Lemma \ref{prod} implies there is a commutative diagram
        \[  \begin{array}[c]{ccc}
      A^4_{(4)}(S\times S)\otimes A^4_{(4)}(S\times S) & \xrightarrow {} &   A^8_{(8)}(S^4)  \\
     \ \ \ \ \ \  \uparrow {\scriptstyle (q^\ast,q^\ast)}&&\ \ \ \ \ \  \uparrow {\scriptstyle (q\times q)^\ast}\\
       A^4_{hom}(Z)\otimes A^4_{hom}(Z) & \xrightarrow {} &   A^8_{}(Z\times Z)  \\
       \end{array}\]
     
  We observe that the right vertical arrow is obviously injective (indeed, $(q\times q)_\ast(q\times q)^\ast=4\ide$). Hence, to prove equality (\ref{Z}),
  it suffices to prove the following statement: for any   
  $a,a^\prime\in A^4_{(4)}(S\times S)$, there is equality
   \begin{equation}\label{SS}   a\times a^\prime - a^\prime\times a =0\ \ \ \hbox{in}\ A^8(S^4)\ .\end{equation}
        
   Since
   \[ A^4_{(4)}(S\times S) = A^2_{(2)}(S)\otimes A^2_{(2)}(S)=      A^2_{hom}(S)\otimes A^2_{hom}(S)    \ ,\]
   we can further reduce (\ref{SS}) to a statement for $S$. This statement for $S$ is known to hold:
   
   \begin{proposition} Let $S$ be a $K3$ surface obtained as (a desingularization of) a double plane branched along the  union of a smooth quartic and a smooth quadric. For any $a,a^\prime\in A^2_{hom}(S)$, one has
   \[ a\times a^\prime - a^\prime\times a =0\ \ \ \hbox{in}\ A^4(S\times S)\ .\]
   \end{proposition}
   
   \begin{proof} This is \cite[Proposition 14]{moi}.
    \end{proof}
   
   This concludes the proof of theorem \ref{main}.
        \end{proof}
  
  As a corollary, a particular case of the generalized Hodge conjecture is verified:
  
  \begin{corollary}\label{ghc} Let $X$ be a Calabi--Yau fourfold as in theorem \ref{main}. Then the Hodge substructure
    \[ \wedge^2 H^4(X)\ \subset\ H^8(X\times X)\]
    is supported on a divisor.
  \end{corollary}
  
  \begin{proof} As noted by Voisin \cite{V9}, this follows from the truth of conjecture \ref{conjvois}. The argument is as follows. 
  First, we note that $X$ satisfies the standard conjecture of Lefschetz type $B(X)$ (because $S\times S$ does, and the two--dimensional centers which are blown--up in the resolution process do). This implies the K\"unneth components $\pi^X_i$ are algebraic \cite{K0}, \cite{K}.
  
  Because $h^{2,0}(X)=0$, the K\"unneth components $\pi^X_2$ and $\pi^X_6$ may be represented by cycles supported on $D\times D$ where $D\subset X$ is a divisor. Also, since $H^3(X)$ injects into $H^3(E)$
 (where $E\subset X$ denotes the exceptional divisor for the resolution morphism $X\to Z$ where $Z$ is a projective quotient variety), and $H^3(E)=N^1 H^3(E)$ (because quotient singularities are rational), we find that $H^3(X)=N^1 H^3(X)$ (here $N^\ast$ denotes the coniveau filtration \cite{BO}). This implies that $\pi^X_3$ is supported on $V\times D$ and $\pi^X_5$ is supported on $D\times V$, where $D\subset X$ is a divisor and $V\subset X$ is of dimension $2$.
  
   Define $\pi^X_4$ as
    \[ \pi^X_4:= \Delta_X - \pi^X_0 - \pi^X_2 -\pi^X_3 -\pi_5^X-\pi^X_6 - \pi^X_8\ \ \ \in A^4(X\times X)\ ,\]
    where $\pi^X_0, \pi^X_8$ are canonically defined (as in \cite{Sch}), and $\pi^X_2, \pi^X_3, \pi^X_5, \pi^X_6$ are as above. (Note that $\pi_1^X$ and $\pi_7^X$ are $0$, because $X$ is Calabi--Yau.)
   For dimension reasons, none of the $\pi^X_j, j\not=4$ act on $A^4_{hom}(X)$ and so
   \begin{equation}\label{only} A^4_{hom}(X)= (\pi^X_4)_\ast A^4_{}(X)\ .\end{equation} 
       
  Let us now  
  define a correspondence
  \[ \Gamma:= {1\over 2}(\Delta_{X\times X} -\Gamma_\tau)\circ (\pi^X_4\times \pi^X_4)\ \ \ \in A^8(X^4)\ ,\]
  where $\tau\colon X\times X\to X\times X$ denotes the involution $(x,y)\mapsto (y,x)$.
  In view of equality (\ref{only}), theorem \ref{main} can be translated as saying that
   \begin{equation}\label{noac} \Gamma_\ast A^8(X\times X)=0\ .\end{equation}
   An argument \`a la Bloch--Srinivas \cite{BS} implies that a correspondence $\Gamma$ with the property (\ref{noac}) has a decomposition
   \[ \Gamma = \Gamma_1 + \Gamma_2\ \ \ \in A^8(X^4)\ ,\]
   where $\Gamma_1$ and $\Gamma_2$ are supported on $D\times X\times X$ resp. on $X\times X\times D$, for some divisor $D\subset X\times X$. 
   
   By construction, $\Gamma$ acts on cohomology as a projector on $\wedge^2 H^4(X)\subset H^8(X\times X)$. To prove the corollary, it only remains to show that
       \begin{equation}\label{both} (\Gamma_j)_\ast H^8(X\times X)\ \subset\ N^1 H^8(X\times X)\ ,\ \ \ j=1,2\ .\end{equation}
   For $\Gamma_2$ this is obvious. For $\Gamma_1$, this is true because the action of $\Gamma_1$ on $H^8(X\times X)$ factors over $H^8(\wt{D})$ (where $\wt{D}$ denotes a resolution of singularities), and $H^8(\wt{D})=N^1 H^8(\wt{D})$ (hard Lefschetz for the sevenfold $\wt{D}$). This proves the inclusion (\ref{both}) for $j=1$, since it is known that the coniveau filtration is preserved by correspondences (\cite[Proposition 1.2]{V4} or \cite{AK}).
    
 \end{proof}

  \vskip1cm
\begin{nonumberingt} 
Many thanks to Yasuyo, Kai and Len for daily pleasant lunch breaks \Winkey.
\end{nonumberingt}

\end{document}